\documentclass[11pt,reqno]{amsart}
\usepackage{amssymb}
\usepackage[curve,matrix,arrow]{xy}
\newtheorem{thm}{Theorem}[section]
\newtheorem{lemma}[thm]{Lemma}
\newtheorem{cor}[thm]{Corollary}
\newtheorem{prop}[thm]{Proposition}

\theoremstyle{definition}
\newtheorem{rem}[thm]{Remark}

\newtheorem{quest}[thm]{Question}

\numberwithin{equation}{thm}
\def\CN{{\mathcal{N}}}

\def\fz{\mathfrak z}

\def\sl{\mathfrak{sl}}

\def\Dim{\operatorname{Dim}\nolimits}

\def\HHH{\operatorname{H}\nolimits}

\def\fP{\mathfrak P}
\def\fv{\mathfrak v}
\def\fg{\mathfrak g}
\def\fu{\mathfrak u}
\def\fn{\mathfrak n}
\def\fa{\mathfrak a}
\def\fb{\mathfrak b}

\title[On the structure of cohomology rings of p-nilpotent Lie algebras]{On the structure of cohomology rings of p-nilpotent Lie algebras (Unabridged Version)}
\author[Jon F. Carlson]{Jon F. Carlson}
\thanks{Research of the first author partially supported by NSF grant
DMS-1001102}
\address{Department of Mathematics, University of Georgia,
Athens, Georgia 30602, USA}
\email{jfc@math.uga.edu}
\author[Daniel K. Nakano]{Daniel K. Nakano}
\thanks{Research of the second author partially supported by NSF grant
DMS-1002135}
\address{Department of Mathematics, University of Georgia,
Athens, Georgia 30602, USA}
\email{nakano@math.uga.edu}

\date\today

\begin{document}

\begin{abstract} 
In this paper the authors investigate the structure the restricted 
Lie algebra cohomology of p-nilpotent Lie algebras with trivial p-power operation.  
Our study is facilitated by a spectral sequence whose $E_{2}$-term 
is the tensor product of the symmetric algebra on the dual of the Lie algebra
with the ordinary Lie algebra cohomology and converges to the 
restricted cohomology ring. In many cases this spectral sequence collapses, and
thus, the restricted Lie algebra cohomology is Cohen-Macaulay. 
A stronger result involves the collapsing of the spectral sequence 
and the cohomology ring identifying as ring with the $E_{2}$-term. 
We present criteria for the collapsing of this spectral sequence 
and provide many examples where
the ring isomorphism fails. Furthermore, we show that there are 
instances when the spectral sequence does not collapse and 
yields cohomology rings which are 
not Cohen-Macaulay. 
\end{abstract}

\maketitle

\section{Introduction} One of the major challenges 
in group cohomology is the computation of the 
cohomology of nilpotent groups. Such computations 
are important because general questions about modular group 
cohomology can often be reduced to questions that 
concern only the cohomology of its 
Sylow $p$-subgroup. The structure of the cohomology of $p$-groups can be quite 
complicated, but in the case when the cohomology ring is Cohen-Macaulay 
(i.e, when the depth equals the Krull dimension), the homological algebra
of the representation theory has more orderly structural features.
In the realm of Lie theory and modular representations 
of algebraic groups the nilpotent restricted $p$-Lie algebras
play a similar role to that of the $p$-groups in group 
representations. For example, an element in the restricted 
cohomology ring of a restricted Lie algebra is nilpotent if and only if
its restriction to every nilpotent subalgebra is nilpotent. 
In this paper we address some 
basic questions on the structure of the cohomology 
rings for these algebras. 

Suppose that $({\fn},[p])$ is a nilpotent restricted 
$p$-Lie algebra and that $k$ is an algebraically closed field of 
characteristic $p>0$. The spectrum of the cohomology 
ring identifies with the restricted nullcone 
${\mathcal N}_{1}({\mathfrak n})=\{x\in {\mathfrak n}:\ x^{[p]}=0\}$. 
If we assume that $p > 2$, then there 
is a spectral sequence \cite{FP}
\[
E_2^{2i,j} = S^{2i}(\fn^*)^{(1)}\otimes \HHH^{j}(\fn,k) \Rightarrow 
\HHH^{2i+j}(u(\fn), k)
\]
where $S^{*}({\fn}^{*})^{(1)}$ is the Frobenius twist of the 
symmetric algebra on the dual of the underlying vector space of $\fn$,
$\HHH^{*}(\fn,k)$ is the ordinary Lie algebra of $\fn$, and 
$\HHH^{*}(u(\fn), k)$, is the cohomology ring of the restricted enveloping algebra 
$u(\fn)$ of $\fn$.  
There are many cases in which it is known that 
the spectral sequence collapses at the $E_2$ page, so that 
$E_2^{*,*}$ is isomorphic to the associated graded 
ring of $\HHH^*(u(\fn),k)$. In this situation, the  
cohomology ring $\HHH^*(u(\fn),k)$ is a free module 
over the symmetric algebra $S^{*}({\fn}^{*})^{(1)}$, 
and the cohomology ring is Cohen-Macaulay. In 1986, 
Friedlander and Parshall \cite{FP} showed this happens 
when $\fn$ is the nilpotent 
radical of the Borel subalgebra of the restricted
Lie algebra of an algebraic group, provided $p>h$ 
where $h$ is the Coxeter number of the associated 
root system. Twenty five years later, Drupieski, Ngo 
and the second author \cite{DNN}
showed that a stronger result holds when $p\geq 2(h-1)$, 
that is, there is a ring isomorphism $\HHH^*(u(\fn), k) 
\cong S^{*}({\fn}^{*})^{(1)}\otimes \HHH^{*}(\fn,k)$. 

The investigations of this paper were originally inspired 
by a computer calculation by the first author 
which demonstrated that if $\fn$ is the Lie 
algebra of the nilpotent radical of a Borel 
subalgebra of a group of type $B_2$ and if $p=5$ (which is 
larger than $h$ but {\em not} larger than $2(h-1)$), then the 
isomorphism $\HHH^*(u(\fn), k) \cong S^{*}({\fn}^{*})^{(1)}\otimes 
\HHH^{*}(\fn,k)$ holds a modules 
over the symmetric algebra $S^{*}({\fn}^{*})^{(1)}$ (so that the 
cohomology ring is Cohen-Macaulay), but it does {\em not}
hold as rings. A search for the reason for this 
phenomenon led to the discovery of a one-dimensional central 
extension of $\fn$ (with trivial $p$th power) 
whose cohomology ring is easily proved to 
be not Cohen-Macaulay. Indeed, this illustrates 
a general situation. One of the theorems in the paper is that if 
$\HHH^*(u(\fn), k) \cong S^{*}({\fn}^{*})^{(1)}\otimes \HHH^{*}(\fn,k)$
holds as an isomorphism of rings, then when $\fn$ is 
replaced by a one-dimensional extension, there is the 
same isomorphism, though perhaps only as an isomorphism of 
$S^{*}({\fn}^{*})^{(1)}$-modules. In the paper, we present numerous examples of this 
phenomenon. 

Cohen-Macaulay rings are of general interest, in part, because they have
very nice structural properties. For example, it can 
be shown that for any restricted Lie algebra, if its cohomology ring is 
Cohen-Macaulay, then the cohomology ring admits a formal 
Poincar\'e duality and its Poincar\'e series, as a rational 
polynomial, satisfies a functional equation \cite{BC1, BC2}. In the
case of a $p$-nilpotent Lie algebra with vanishing $p$-power
operation, we show that the cohomology ring is Cohen-Macaulay 
if and only if the spectral sequence 
given above collapses at the $E_2$-page. 

The paper is organized as follows. In the following 
section of the paper, we present 
preliminaries about cohomology rings and the 
definitions from commutative algebra.
A proof that the cohomology ring is Cohen-Macaulay 
if and only if the spectral sequence collapses is given 
in Section 3. We also prove results which 
describe how the spectral sequence behaves under 
central extensions in that section. The next 
four sections are concerned with specific examples. In the case that if $\fn$ 
is the nilpotent radical of a Borel 
subalgebra of $\sl_3$ (i.e., type $A_2$), and the 
field has characteristic 3, then the isomorphism 
$\HHH^*(u(\fn), k) \cong
S^{*}({\fn})^{(1)}\otimes \HHH^{*}(\fn,k)$
holds as modules over the symmetric algebra, but not as ring. 
Similar examples are given in type $B_2$ in characteristic 5 and in 
type $G_2$ in characteristic 7. In each example, there is a one-dimensional
extension of the Lie algebra whose cohomology ring is not Cohen-Macaulay. 
Examples of Lie algebra whose cohomology rings are not Cohen-Macaulay
are given in all characteristics. In Section 8, we 
outline which Lie algebra of dimension 5 have 
cohomology rings that are Cohen-Macaulay. In Section 9, we look at the 
special case of the nilpotent radical of the Borel 
subalgebra of $\mathfrak{sl}_4$ when $h<p<2(h-1)$ and 
show that the cohomology ring can be identified the 
symmetric algebra tensored with the ordinary Lie algebra cohomology as rings. 

\section{Preliminaries} 

Let $({\mathfrak g},[p])$ be a restricted Lie algebra 
over an algebraically closed field of characteristic 
$p>0$. Throughout this paper we will work with 
the added assumption that $p\geq 3$. 
The restricted representations for ${\mathfrak g}$ 
correspond to modules for the restricted enveloping 
algebra $u({\mathfrak g})$. Since $u({\mathfrak g})$ 
is a finite-dimensional cocommutative Hopf 
algebra the cohomology ring 
$\text{H}^{*}(u({\mathfrak g}),k)$ is a finitely 
generated graded-commutative $k$-algebra. 
For these types of rings, notions like 
Krull dimension and spectrum are well defined. 
The spectrum of the cohomology ring 
${\mathcal V}_{\mathfrak g}$ is homeomorphic to 
the restricted nullcone: 
\[
{\mathcal N}_{1}({\mathfrak g}):=
\{x\in {\mathfrak g}:\ x^{[p]}=0\}.
\]

\vskip.1in
When $p\geq 3$ there exists a the spectral
sequence 
\begin{equation}\label{specseq1}
E_2^{2i,j} \ \ = \ \ S^i({\mathfrak g}^*)^{(1)} 
\otimes \HHH^i({\mathfrak g},k) \Rightarrow
\HHH^{2i+j}(u({\mathfrak g}), k) 
\end{equation}
where $\text{H}^{*}({\mathfrak g},k)$ is the 
ordinary Lie algebra cohomology. 
In particular, the map of the universal 
enveloping algebra of $\fg$ to the restricted 
enveloping algebra of $\fg$ induces an edge 
homomorphism from the restricted cohomology 
to the ordinary Lie 
algebra cohomology. On the other hand, 
there is another edge 
homomorphism $\Phi: S^{*}({\fg}^{*})^{(1)}\rightarrow 
\HHH^{2*}(u({\fg}),k)$ which induces 
an inclusion of 
$R=k[{\mathcal N}_{1}({\fg})]\hookrightarrow 
\HHH^{*}(u({\fg}),k)$. Furthermore,  
$\HHH^{*}(u({\fg}),k)$ is an integral 
extension of $R$. 

The following observation is a consequence of these facts in the 
event that the $p$-power operation vanishes on a 
nilpotent restricted $p$-Lie algebra $\fn$. In this situation, $\CN_1(\fn)
= \fn$ and its coordinate ring is $S^*(\fn^*)^{(1)}$.  Indeed, any 
restricted Lie algebra with vanishing $p$-power operation is nilpotent. 

\begin{prop} \label{prop:symmetric}
Suppose that $\fn$ is a restricted  Lie algebra such 
that $x^{[p]} = 0$ for all $x \in \fn$. Then the edge 
homomorphism $S^*(\fn^*)^{(1)} \to \HHH^{2*}(u(\fn),k)$
is an injection. 
\end{prop}

We require several ring theoretic notions which, though usually 
defined for commutative rings or commutative local rings, 
apply also to graded-commutative $k$-algebras. 

Suppose that 
$A = \sum_{i \geq 0} A_i$ is a graded-commutative $k$-algebra and that 
$M = \sum_{i \geq 0} M_i$ is a graded $A$-module. A sequence
$x_1, \dots, x_r$ of homogeneous elements of $A$ is said to be 
a {\it regular sequence} for $M$ if for every $i = 1, \dots, r$, we
have that multiplication by $x_i$ is an injective map from 
$M/(x_1, \dots, x_{i-1})M$ to itself. The {\it depth} of $M$ is 
the length of the longest regular sequence for $M$, and the 
depth of $A$ is the depth of $A$ as a module over itself. 

A sequence of homogeneous elements $x_1, \dots, x_r$ is a homogeneous 
set of parameters for $A$, if $x_1, \dots, x_r$
generate a polynomial subring $R = k[x_1, \dots, x_r] \subseteq A$ 
and that $A$ is finitely generated as a module over $R$. In this case, 
the number $r$ must be the Krull dimension of $A$. The module 
$M$ is {\it Cohen-Macaulay} if its depth is equal to the Krull 
dimension of $A$. The algebra $A$ is Cohen-Macaulay if it is 
Cohen-Macaulay as a module over itself. This is equivalent to
the condition that there be a homogeneous set of parameters 
$x_1, \dots, x_r$ for $A$ such that $A$ is a finitely generated
free module over the polynomial subring $k[x_1, \dots, x_r]$.  
It is a theorem that if there is a homogeneous set of parameters
$x_1, \dots, x_r$ such that $A$ is a finitely generated free
module over $k[x_1, \dots, x_r]$, then $A$ is a finitely generated
free module over $k[y_1, \dots, y_r]$ for any homogeneous
set of parameters $y_1, \dots, y_r$. For a reference, 
see Proposition 3.1 of \cite{St} or Theorem 2, page IV-20 of 
\cite{Se}. Note that in the book of Serre, the proof is given
only for commutative local rings, but can be easily adapted to 
the graded-commutative case. 

With these preliminaries, we can extract the results 
that we need for this paper. 

\begin{thm}\label{thm:CM-prelim}
Suppose that $\fn$ is a $p$-nilpotent Lie algebra with trivial 
$p$th-power operation (i.e., $x^{[p]} = 0$ for all $x \in \fn$). 
Then the cohomology ring $\HHH^*(u(\fn),k)$
is Cohen-Macaulay if and only if it is a free module over 
the polynomial subring $S^{*}(\fn^*)^{(1)}$. In particular, 
no nonzero element of $S^{*}(\fn^*)^{(1)}$ can be a divisor of 
zero if the cohomology ring is Cohen-Macaulay. 
\end{thm}

\begin{proof} The last statement clearly follows from the first part of the 
theorem. Suppose that $x_1, \dots, x_r$ is a basis for $\fn^*$.
By Proposition~\ref{prop:symmetric}, 
$$S^{*}(\fn^*)^{(1)}= k[x_1, \dots, x_r] \subseteq \HHH^*(u(\fn),k).$$
Because the ordinary Lie algebra cohomology $\HHH^*(\fn,k)$
is finite dimensional, we have that $x_1, \dots, x_r$
is a homogeneous set of parameters for $\HHH^*(u(\fn),k)$. 

If $\HHH^*(u(\fn),k)$ is a finitely generated 
free module over $S(\fn^*)^{(1)}$, then 
it is Cohen-Macaulay. On the other hand, if $\HHH^*(u(\fn),k)$
is Cohen-Macaulay, then it must be free as a 
module over $S(\fn^*)^{(1)}$, by the results 
cited above. 
\end{proof}

\section{Consequences of the Lie cohomology spectral sequence}

Let $\fn$ be a restricted Lie algebra with trivial $p$-restriction. 
We consider the first quadrant spectral sequence given as 
\begin{equation} \label{eq:specseq}
E_2^{2i,j} \ \ = \ \ S^i(\fn^*)^{(1)} 
\otimes \HHH^i(\fn,k) \Rightarrow
\HHH^{2i+j}(u(\fn), j) 
\end{equation}
One of important fact to note about the spectral sequence
is that $E_2^{i,j} = \{0\}$ if $j > \Dim(\fn)$. This is simply because
the ordinary Lie algebra cohomology of $\fn$ has the property that 
$\HHH^j(\fn, k) = \{0\}$ for $j > \Dim(\fn)$. 

This spectral sequence can be used to show that 
the cohomology of $u({\mathfrak n})$ is Cohen-Macaulay. 

\begin{prop}\label{prop:specseq}
If the spectral sequence $E_*^{*,*}$ collapses at the $E_2$ page,
then $\HHH^*(u(\fn),k)$ is free as module over the polynomial 
subalgebra $S= S^{*}(\fn^*)^{(1)}.$  In particular, it is 
Cohen-Macaulay. 
\end{prop}

\begin{proof}
The spectral sequence is a filtered version of the cohomology 
$\HHH^*(u({\mathfrak n}),k)$. That is, if $\zeta \in E_2^{i,j}$ and $\eta
\in E_2^{r,s}$ then $\zeta\eta \in \sum_{\ell \geq 0} E_2^{i+r+\ell,
j+s-\ell}$. For this reason, for any $m$ the collection of 
the lowest $m$ rows, ($U_m = \sum E_2^{i,j}$ with $i \geq 0$ and   
$0 \leq j \leq m$) is a module over $S$, which lies in the 
bottom row. Because the spectral sequence collapses, 
$E_2 = E_{\infty}$ and the 
quotients $U_m/U_{m-1}$ are free $S$-modules.  Hence, 
the proposition follows from the fact that every one of the 
quotient maps $U_m \to U_m/U_{m-1}$ must split as a map of 
$S$-modules. 
\end{proof}

Many of the results of 
\cite{BC1} apply in the case that the polynomial ring is 
Cohen-Macaulay. In particular, we have the following 
adaptation of \cite[Theorem 1.1]{BC1}. We refer the reader
to that paper for the proof which carries over from group 
cohomology to restricted Lie algebra cohomology with only
minimal changes.

\begin{thm} \label{thm:bc}
Suppose that $\fn$ is a restricted $p$-Lie algebra with trivial $p$th-power
operation. Let $d$ denote the dimension of $\fu$. 
If the cohomology ring $\HHH^*(u(\fn), k)$
is Cohen-Macaulay, then any basis for $\fn^*$, meaning
any complete linearly independent set of 
degree-two generators $X_1, \dots, X_d$ for the symmetric algebra
$S^{*}(\fn^*)^{(1)}$, is a homogeneous set of parameters for 
$\HHH^*(u(\fn),k)$ and the quotient 
\[
\HHH^*(u(\fn),k)/(X_1, \dots, X_d)
\]
satisfies Poincar\'e duality in formal dimension $d$. Moreover, 
the Poincar\'e series $P_k(t) = \sum_{i \geq 0} 
\Dim \HHH^i(u(\fu),k)\ t^{i}$, regarded as a rational function of $t$
satisfies the functional equation
\[
P_k(1/t) = (-t)^d P_k(t).
\]
\end{thm}

A consequence of the preceding result is the following. 

\begin{cor}\label{cor:bc}
Suppose that $\fn$ is a restricted $p$-Lie algebra of dimension $d$
whose $p$-power operation is trivial. Then the 
spectral sequence (\ref{eq:specseq})  collapses at the $E_2$ page if and only if 
the cohomology ring $\HHH^*(u(\fn),k)$ is Cohen-Macaulay.
In this case, the edge homomorphism $\HHH^*(u(\fn),k) \to \HHH^*(\fn,k)$ 
is surjective and the Poincar\'e series, as a rational 
function has the form 
\[
P_k(t) = \frac{f_k(t)}{(1-t^2)^d},
\]
where $f_k(t)$ is the Poincar\'e polynomial for the ordinary
Lie algebra cohomology $\HHH^*(\fn,k)$. 
\end{cor}

\begin{proof}
For convenience of notation let $S = S^{*}(\fn^*)^{(1)}$ and let 
$\HHH^* = \HHH^*(u(\fn).k)$. By Proposition \ref{prop:specseq},
if the spectral sequence (\ref{eq:specseq}) collapses at the $E_2$
page then $\HHH^*$ is Cohen-Macaulay. We need to prove the 
converse. So assume that $\HHH^*$ is Cohen-Macaulay. For $t \geq 0$,
let $M_t = S \cdot \sum_{j=0}^t \HHH^j$, the $S$-submodule of 
$\HHH^*$ generated by elements of degree at most $t$. By Theorem
\ref{thm:bc}, this is a free $S$-module. That is, let $I = (X_1, 
\dots, X_d)$ (in the notation of the theorem). Then $\HHH^*$ 
is a free $S$-module on a set of homogeneous elements 
$\zeta_1, \dots, \zeta_m$ whose classes form a basis of $\HHH^*/I$.
If $\zeta_1, \dots, \zeta_\ell$ are all of those element of 
degree at most $t$, then it is easily seen that $M_t$ is a 
free module on these elements. 

Now we proceed by induction on $t$ to prove that 
\[
M_t/I \cdot M_t \ \cong \ \sum_{i = 0}^t E_2^{0,i},
\]
as vector spaces, and that no differential $d_r$ on the 
$r^{th}$ page of the spectral 
sequence has a nonzero image on any of the lines $E_r^{*,0}, \dots,
E_r^{*,t}$.  This is true if $t = 0$, by Proposition \ref{prop:symmetric}.
That is, the differential $d_r$ on the $r^{th}$-page cannot have a 
nonzero image $d_r:E_r^{j,r-1} \to E_r^{j+r,0}$ as otherwise the 
edge homomorphism onto the bottom row would not be injective. 

So assume that the statement is true for a certain value of $t$.
Then the differential $d_r$ on the $r^{th}$ page of the spectral 
sequence must vanish on $E_r^{*,t+1}$, as otherwise there would be 
a nonzero image on one of the lower lines. Therefore, $E_2^{0,j} 
= E_\infty^{0,j}$ for all $j$ with $0 \leq j \leq t+1$. As a consequence,
$M_{t+1}$ is generated as an $S$-module by elements representing a 
$k$-basis for $\sum_{i=0}^{t+1} E_\infty^{0,i} \cong \sum_{i=0}^{t+1}
E_2^{0,i}$. That is, $M_{t+1} \cong \sum_{i=1}^{t+1} E_\infty^{*,i}$.
Because, this is a free $S$-module, it must be that 
\[
\sum_{i=1}^{t+1} E_\infty^{*,i} \ \cong \ \sum_{i=1}^{t+1} E_2^{*,i}
\]
since these are both free modules on the same number of generators. 
It follows that no differential of the spectral sequence can have a
nonzero image on row $t+1$. The induction proves the corollary. 
\end{proof}

We record the following lemma which will be later 
useful in comparing spectral sequences. 

\begin{lemma} \label{lem:eqdim}
Suppose that $\Dim \HHH^m(u(\fn),k) = \sum_{2i+j = m} 
\Dim(S^{2i}(\fn^*) \otimes \HHH^j(\fn,k))$ for 
all $m \geq 0$.  Then the spectral sequence (\ref{eq:specseq})
collapses at the $E_2$ page and the cohomology ring $\HHH^*(u(\fn),k)$
is Cohen-Macaulay.  
\end{lemma}

\begin{proof}
The hypotheses of the lemma asserts that $\Dim \HHH^m(u({\mathfrak n}),k) =
\sum_{m = 2i+j} \Dim E_2^{2i+j}$. The condition forces the 
spectral sequence to collapse at the $E_2$ page, because 
otherwise there would be some further nontrivial differential
that would reduce the dimension.
\end{proof}

Nilpotent Lie algebras can be built up from central extensions. The 
next theorem provides conditions 
on when the spectral sequence will 
collapse at $E_{2}$ under a central extension and 
yield an isomorphism of $S^{*}({\mathfrak n}^{*})^{(1)}$-modules. 

\begin{thm}\label{thm:tensoralg}
Let $\fn$ be a nilpotent restricted $p$-Lie algebra with
trivial $p$-power operation. Assume that $\fz$ is a central ideal 
of dimension one. Suppose that we have
an isomorphism of rings
\[
\HHH^*(u(\fn/\fz), k) \cong S^{*}((\fn/\fz)^*)^{(1)} 
\otimes \HHH^*(\fn/\fz,k).
\]
Then 
\[
\HHH^*(u(\fn), k) \cong S^{*}(\fn^*)^{(1)} 
\otimes \HHH^*(\fn,k).
\]
as modules over $S^{*}(\fn^*)^{(1)}$. In particular,
$\HHH^*(u(\fn),k)$ is Cohen-Macaulay. 
\end{thm}

\begin{proof}
We use the Lyndon-Hochschild-Serre (LHS) spectral sequence 
\[
E_2^{i,j} = \HHH^i(u(\fn/\fz), \HHH^j(u(\fz),k)) 
\Rightarrow  \HHH^{i+j}(u(\fn),k).
\]
Since ${\mathfrak z}$ is central,  $\fn$ acts trivially on $\fz$
and hence also trivially on its cohomology. Thus we have 
$E_{2}^{i,j}\cong \HHH^i(u(\fn/\fz),k)\otimes \HHH^j(u(\fz),k)$.  

Next note that $E_2^{0,1}$ has dimension one, and is spanned by an 
element $x$. Then $d_2(x) = s + v$, where $s$ is in 
$S^1((\fn/\fz)^*)^{(1)} \otimes 1$ and $v$ is in 
$1 \otimes \HHH^2(\fn/\fz,k)$ by the hypothesis. Also 
by the hypothesis, we have that $v^n = 0$ for $n$ sufficiently
large, since $\HHH^*(\fn/\fz,k)$ has finite dimension. 
Consequently, for $r$ sufficiently large we have that
$0 = (s+v)^{p^r} = s^{p^r} + v^{p^r}= s^{p^r}$ on the 
$E_3$ page of the spectral sequence. However, if $s$ is
not zero, then we have a contradiction to the fact that
$S^{*}(\fn^*)^{(1)}$ injects into the cohomology ring
$\HHH^{2*}(u(\fn),k)$. So $d_2(x) \in 1 \otimes 
\HHH^2(\fn/\fz,k)$.

Next we see that the rows $E_2^{*,0}$ and $E_2^{*,1}$ are 
both isomorphic to $\HHH^*(u(\fn/\fz), k)$ as modules
over the symmetric algebra $S^{*}((\fn/\fz)^*)^{(1)}$. Moreover, 
the differential $d_2:  E_2^{*,1} \to E_2^{*+2,0}$
is a homomorphism of $S^{*}((\fn/\fz)^*)^{(1)}$-modules. More 
specifically, we have that the differential 
\[
d_2 : S^{*}(\fn^*)^{(1)} \otimes \HHH^{*}(\fn,k) \cong E_2^{*,1} 
\longrightarrow E_2^{*,0} \cong S^{*}(\fn^*)^{(1)} \otimes \HHH^{*}(\fn,k)
\]
is multiplication by $d_2(x)$ which has the form 
$d_2(x) = 1 \otimes w \in 1 \otimes 
\HHH^2(\fn/\fz,k)$. Hence, by the hypothesis, we have
that $E_3^{*,1} \cong S(\fn^*)^{(1)} \otimes K$ and
$E_3^{*,1} \cong S(\fn^*)^{(1)} \otimes C$, where 
$K$ and $C$  are 
respectively the kernel and cokernel of multiplication
by $w$ on $\HHH^*(\fn/\fz,k)$. 

We now observe that the element $w \in \HHH^2(\fn/\fz, k)$ 
is the extension class associated to the extension 
$\fz \hookrightarrow \fn \rightarrow \fn/\fz$. In the LHS 
spectral sequence $\hat{E}_2^{i,j} = \HHH^i(\fn/\fz, 
\HHH^j(\fz,k)) \Rightarrow \HHH^{i+j}(\fn,k)$, 
of ordinary Lie algebra cohomology associated to that
sequence, the differential $d_2: \hat{E}_2^{*,1}
\to \hat{E}_2^{*,0}$ is multiplication by $w$. In 
addition, there can be no further differentials 
in that spectral sequence, since the sequence has
only two nonzero rows. It 
follows that $E_2^{*,0} \oplus E_2^{*,1} \cong 
S^{*}((\fn/\fz)^*)^{(1)} \otimes \HHH^*(\fn, k)$ as 
(free) modules over $S^{*}((\fn/\fz)^*)^{(1)}$. 

Finally, we note that $E_3 = E_\infty$. The reason is that
the $E_3$ page is generated, as a ring by the elements on 
the bottom two rows ($E_3^{*,0}$ and $E_3^{*,1}$) and 
an element $X$ in $E_3^{0,2}$ that represents the class
of a generator in $S^{*}(\fz^*)^{(1)} \subseteq S^{*}(\fn^*)^{(1)}$.
The class $X$ must survive until the $E_\infty$ page
of the spectral sequence.  So $d_2(X) = 0$, and we must have
that $XE_3^{i,j} = E_3^{i, j+2}$ for all $i, j \geq 0$.
Consequently, the differential
$d_3$ must vanish, because it vanishes on a collection of 
ring generators. The same holds for all further 
differentials in the spectral sequence. 

We have verified the hypothesis of Lemma \ref{lem:eqdim}
and the theorem is proved. 
\end{proof}

As an initial application we offer the following.  Notice that the 
algebra $\fn$ in the corollary must be the direct sum of a commutative
Lie algebra and a Heisenberg Lie algebra. Assuming $p \geq 3$, any such
ordinar Lie algebra can be made into a restricted Lie algebra by 
assuming a vanishing $p$-power operation. 

\begin{cor} \label{cor:dim1comm}
Suppose that $\Dim([\fn, \fn]) = 1$. 
Then $\HHH^*(u({\mathfrak n}), k) \cong S^{*}(\fn^*)^{(1)} \otimes 
\HHH^*(\fn,k)$, as $S^{*}(\fn^*)^{(1)}$-modules,
and $\HHH^*(u({\mathfrak n}),k)$ is Cohen-Macaulay. 
\end{cor}

\begin{proof} Observe that the quotient algebra $\fv = \fn/[\fn,\fn]$ 
is commutative and hence we have that 
$\HHH^*(u(\fv), k) \cong S^{*}(\fv^*)^{(1)} 
\otimes \HHH^*(\fv,k)$ as rings. Thus, Theorem~\ref{thm:tensoralg}    
implies the corollary. 
\end{proof}

The next theorem provides 
stronger conditions which will show when one can identify 
$\HHH^{*}(u({\fn}),k)$ with $S^{*}({\fn}^{*})^{(1)}\otimes 
\HHH^{*}({\fn},k)$ as rings. Together this theorem in conjunction with Theorem~\ref{thm:tensoralg} 
can be applied to inductively compute cohomology rings. 

\begin{thm}\label{th:splitting} Let ${\fn}$ be a 
$p$-nilpotent Lie algebra and suppose that 
there is an isomorphism of $S^{*}({\fn}^{*})^{(1)}$-modules, 
$$
\HHH^{*}(u({\fn}),k)\cong S^{*}({\fn}^{*})^{(1)}
\otimes \HHH^{*}({\fn},k).
$$ 
Moreover, assume that there exists a subalgebra $B$ in 
$\HHH^{*}(u({\fn}),k)$ such that 
$B\cong \HHH^{*}({\fn},k)$ under the map 
$\phi:\HHH^{*} (u({\fn}),k) \rightarrow 
\HHH^{*}({\fn},k)$. Then 
$\HHH^{*}(u({\fn}),k)\cong S^{*}({\fn}^{*})^{[1]}\otimes 
\HHH^{*}({\fn},k)$ as rings. 
\end{thm} 

\begin{proof} Let $A$ be the subalgebra in 
$\HHH^{*}(u({\fn}),k)$ isomorphic to $S^{*}({\fn}^{*})^{(1)}$. 
We have an algebra homomorphism $\Gamma$ defined by 
\[
S^{*}({\fn}^{*})^{(1)}\otimes \HHH^{*}({\fn},k)
\rightarrow A\otimes B \rightarrow \HHH^{*}(u({\fn}),k)\otimes 
\HHH^{*}(u({\fn}),k) \rightarrow \HHH^{*}(u({\fn}),k).
\] 
The last map is given by the cup product. This map is bijective because  
\[
\HHH^{*}(u({\fn}),k)\cong S^{*}({\fn}^{*})^{(1)}\otimes 
\HHH^{*}({\fn},k).
\] 
as $S^{*}({\fn}^{*})^{(1)}$-modules. 
\end{proof} 

%%%%%%%%%%%%%%%%%%%%%  Section on Examples A %%%%%%%%%%%%%%%%%%%%

\section{Some examples of type A}
We begin with the example of the nilpotent radical $\fn$ of a Borel
subalgebra of $\sl_3$. The relations for the cohomology in 
characteristic 3 were 
calculated by computer using the system Magma \cite{BoCa}. 
Specifically, we use the package for basic algebras written by
the first author. Two of the three unusual relation can be 
derived from the second example in this section. 

Note that $\fn$ has an action of a two dimensional torus. 
Let $\alpha$ and $\beta$ be the simple roots. By convention 
the weights of $\fn$ consist of sums of negative roots so that its 
cohomology has weights in the positive cone of roots. For convenience, we 
subscript elements by their weights whenever this causes no 
problems.  

\begin{lemma}\label{sl3-ordin}
Let $\fn$ be the nilpotent radical of a Borel subalgebra of 
$\sl_3$ over a field of characteristic at least three. 
Then the ordinary Lie algebra cohomology of $\fn$ is 
given as 
\[
\HHH^*(\fn, k) = k[\eta_\alpha, \eta_\beta, \eta_{2\alpha+\beta},
\eta_{\alpha+2\beta}]/I
\]
where $I$ is the ideal generated by 
\[
\eta_\alpha^2, \ \eta_\alpha\eta_\beta, \
\eta_\beta^2, \ \eta_\alpha\eta_{2\alpha+\beta}, \ 
\eta_\beta\eta_{\alpha+2\beta}, \ 
\eta_\beta\eta_{2\alpha+\beta} +\eta_{\alpha}\eta_{\alpha+2\beta}, \
\eta_{2\alpha+\beta}^2, \
\eta_{\alpha+2\beta}^2,   \ 
\eta_{2\alpha+\beta}\eta_{\alpha+2\beta} \ 
\] 
\end{lemma}

\begin{proof}
The result follows easily from the LHS spectral sequence 
$E_2^{i,j} = \HHH^i(\fn/\fz, \HHH^j(\fz,k))
\Rightarrow \HHH^{i+j}(\fn,k)$. Note that the 
torus $T$ acts on the spectral sequence. We let
$\eta_\alpha$ and $\eta_\beta$ be the generators of $E_2^{1,0}$
having weights $\alpha$ and $\beta$ and let 
$\eta_{\alpha+\beta}$ be the generator of $E_2^{0,1}$ with weight
$\alpha+\beta$. It is easy to check
that $d_2(\eta_{\alpha+\beta}) = 
\eta_\alpha\eta_\beta$, which is the extension class. The elements 
$\eta_\alpha\eta_{\alpha+\beta}$ and 
$\eta_{\beta}\eta_{\alpha+\beta}$ survive to 
the $E_3 = E_{\infty}$ page, and are ring generators \-- 
which we call $\eta_{2\alpha+\beta}$ and 
$\eta_{\alpha+2\beta}$, respectively. The
relations follow easily from the relations 
in the exterior algebra of ${\mathfrak n}^{*}$. 
\end{proof}

With this lemma, we can compute the cohomology of the 
restricted Lie algebra. The following calculation is 
computer generated in part.  However,
as we see in the next example all but one of the 
relations can be derived by hand. 

\begin{prop}\label{sl3-rest}
Let $k$ be a field of characteristic 3. 
Let $u(\fn)$ be the restricted enveloping algebra 
of the Lie algebra $\fn$ which is the nilpotent 
radical of the Borel subalgebra of $\sl_3$. Then the 
cohomology ring $\HHH^{*}(u(\fn),k) \cong S^{*}(\fn^*)^{(1)} 
\otimes \HHH^*(\fn, k)$ is a free 
$S^{*}(\fn^*)^{(1)}$-module with basis consisting 
of the images of a basis of the ordinary Lie algebra
cohomology of $\fn$ as in Lemma \ref{sl3-ordin}.
The ring $S^{*}(\fn^*)^{(1)}$ is a polynomial ring 
in variables $X_\alpha, X_\beta$ and 
$X_{\alpha+\beta}$ having weights
$3\alpha, 3\beta$ and $3(\alpha+\beta)$ under
the action of the torus. 
The multiplicative relations are given by 
\[
\eta_\alpha^2, \ 
\eta_\alpha\eta_\beta, \
\eta_\beta^2, \ 
\eta_\beta\eta_{2\alpha+\beta} +\eta_{\alpha}\eta_{\alpha+2\beta}, \
\eta_{2\alpha+\beta}^2, \
\eta_{\alpha+2\beta}^2,   \ 
\]
\[
\eta_\alpha\eta_{2\alpha+\beta}- \eta_\beta X_\alpha, \ 
\eta_\beta\eta_{\alpha+2\beta} - \eta_\alpha X_\beta, \ 
\eta_{2\alpha+\beta}\eta_{\alpha+2\beta} - X_\alpha X_\beta
\]
\end{prop}

\begin{proof}
First note that the isomorphism 
$\HHH^{*}(u(\fn),k) \cong S^{*}(\fn^*)^{(1)} 
\otimes \HHH^*(\fn, k)$ is a consequence of 
Corollary \ref{cor:dim1comm}.
Consider that LHS spectral sequence 
\[
E_2^{i,j} = \HHH^i(u(\fn/\fz), \HHH^j(u(\fz),k)) 
\Rightarrow  \HHH^{i+j}(u(\fn),k)
\]
We have elements $a,b$ in $E_2^{1,0}$ and $u$ in $E_2^{0,1}$.
The differential on the $E_2$ page has the form $d_2(u) = ab$
as in the proof of Lemma~\ref{sl3-ordin}. A representative of 
$X_{\alpha+\beta}$ is in $E_2^{0,2}$, and this element must live until
the $E_{\infty}$ page. Because the resulting ring at the 
$E_3$ page is generated in degrees one and two, we conclude
that all further differentials must vanish and 
$E_3 = E_\infty$. The problem is to ungrade the spectral
sequence. The first six relations are forced by the 
grading on the spectral sequence and by the action of
the torus. For example,
the relation 
$\eta_\beta\eta_{2\alpha+\beta} +
\eta_{\alpha}\eta_{\alpha+2\beta}$ holds on the 
(graded) $E_3$ page and must hold in the ungrading 
because there is no nonzero element of that 
weight ($2\alpha+2\beta$) in $E_3^{3,0}$. 

For the last three relations, we rely on the computer. 
For example, $\eta_\alpha\eta_{2\alpha+\beta} = 0$ in the graded
$E_3$ page. This element lies in $E_3^{2,1}$ in the 
ungrading it is equal to $\eta_\beta X_\alpha \in E_3^{4,0}$. 
Note that both elements
have weight $3\alpha+\beta$. The other two relations
are similar. 
\end{proof}

Next we extend the example slightly to get a 
nilpotent Lie algebra (with trivial $p$th power) 
where the restricted cohomology fails to 
be Cohen-Macaulay. We consider the algebra $\fn$, labeled
as $L_{5,9}$ in the list of de Graaf \cite{deG}.
This algebra is isomorphic to the quotient of the
Lie algebra of all upper triangular $4 \times 4$ 
matrices by its center. The algebra can also be represented
as the algebra of all strictly upper triangular matrices 
such that the all entries in the third (and fourth) row 
are zero. Notice that in this representation,  the algebra has trivial
$p$-power operation.  The algebra has a 
basis consisting of $u_1, \dots, u_5$, where 
$u_4$ and $u_5$ are central, and with 
the additional relations:
\[
[u_1, u_2] = u_4, [u_2, u_3] = u_5, 
[u_1, u_3] = 0.  
\] 
The reader should be aware that this is a minor change 
from the presentation in \cite{deG}. 

\begin{prop}\label{prop:notcm1}
Suppose that the characteristic of $k$ is 3, 
and let $\fn$ be as above. The cohomology ring 
$\HHH^*(u(\fn),k)$ is not Cohen-Macaulay. In
particular, it has an associated prime $\fP$ 
such that $\HHH^*(u(\fn),k)/\fP$ has Krull 
dimension four. 
\end{prop}

\begin{proof}
First we observe that the algebra has an action
of a three dimensional torus, $T$ (in the representation
as upper triangular $4 \times 4$ matrices modulo the 
center of that algebra). With this action, 
the basis elements $u_1, \dots, u_5$ have weights
$-\alpha, \ -\beta, \ -\gamma, \ -\alpha-\beta$ and $-\beta-
\gamma$, respectively. The torus also acts on the 
cohomology. 

Let $\fv$ be the commutative subalgebra of $\fn$ spanned
by $u_1, u_3, u_4, u_5$. Its Lie algebra cohomology 
is an exterior algebra generated by elements 
$\eta_\alpha, \eta_{\alpha+\beta}, 
\eta_{\beta+\gamma}, \eta_\gamma$ 
all in degree one and  having weights 
as indicated by the subscripts. The element $u_2$ acts
on $\fv$ and on its cohomology. After possible 
rescaling, we have that $u_2\eta_{\alpha+\beta} = 
\eta_\alpha$ and  $u_2\eta_{\beta+\gamma} =
\eta_\gamma.$  Recall that the action on the
cohomology is dual to the action on the algebra,
so multiplying by $u_2$ on an element of 
cohomology subtracts the root $\beta$.

Now we consider the spectral sequence 
\[
E_2^{i,j} = \HHH^i(u(\fn/\fv), \HHH^j(u(\fv),k) 
\Rightarrow \HHH^{i+j}(u(\fn), k).
\]
As a module over $u(\fn/\fv)$, 
\[
\HHH^1(u(\fv))  \ \cong \ M_1 \oplus M_2
\]
where $M_1$ is the span of 
$\{\eta_\alpha, \eta_{\alpha+\beta} \}$,
$M_2$ is the span of 
$\{ \eta_\gamma, \eta_{\beta+\gamma} \}$ and the 
action  of the class of $u_2$ is given as above. 
In degree 2, we have that 
\[
\HHH^2(u(\fn), k) \ \cong \Lambda^{2}(M_1) \ \
\oplus \ \ M_1 \otimes M_2 \ \  \oplus \ \ 
\Lambda^{2}(M_2)  \ \ \oplus \ \ k^4
\]
where the last factor is spanned by the 
generators $X_\alpha, X_\gamma, X_{\alpha+\beta}, X_{\beta+\gamma}$ 
(each having weight equal to three times its index) of 
$S^{*}(\fv^*)^{(1)}$ that are fixed by the action
of $u_2$. The interesting part of this is 
the tensor product $M_1 \otimes M_2$ which
is the direct sum of a trivial module 
spanned by $\eta_\alpha \wedge \eta_{\beta+\gamma} - 
\eta_{\alpha+\beta} \wedge \eta_\gamma$
and an indecomposable three dimensional 
module generated by $\eta_{\alpha+\beta} \wedge 
\eta_{\beta+\gamma}$ and
with socle spanned by $w = \eta_\alpha \wedge \eta_\gamma$.
Because the characteristic is 3, this is a
free module over $u(\fn/\fv)$. Thus there 
is an element $E_2^{0,2} \cong
\HHH^2(u(\fv),k)^{\fn/\fv}$ that is 
determined by $w$. Write $M_1 \otimes M_2 
\cong k \oplus N$, where $N$ is the free 
submodule with socle spanned by $w$. 

Let $X_\beta \in E_2^{2,0}$
be the generator of $S^{1}((\fn/\fv)^{*})^{(1)}$. This
elements survives to the $E_\infty$ page. 
Moreover, every element in $E_2^{i,j}$
for $j \geq 2$ is a multiple of this element. 
We claim that this element must be contained
in an associated prime of $\HHH^*(u(\fn), k)$. 
Specifically, the element $\hat{w} \in 
\HHH^0(u(\fn/\fv),N)$ 
determined by $w$ in $E_2^{2, 0}$ 
has the property that $X_\beta\hat{w} = 0$
on the $E_2$ page because $N$ is free 
over $u(\fn/\fv)$. So we only need to 
show that $\hat{w}$ survives to the $E_\infty$
page and that the product $X_\beta\hat{w}$
does not ungrade to something that is nonzero. Both 
of these statements can be deduced from
looking at the action of the torus. 
That is, $w$ has weight $\alpha + \gamma$ 
and there is no element of that weight in
either $E_2^{2,1}$ or $E_2^{0,3}$. Hence 
both $d_2$ and $d_3$ must both vanish on
the class of $w$. Likewise, $X_\beta\hat{w}$
has weight $\alpha +3\beta +\gamma$ and
there is no element of that weight in 
$E_\infty^{3,1}$ or $E_\infty^{4,0}$.
So we must have that $X_\beta$ annihilates
the class of $w$ in $\HHH^*(u(\fn),k)$
\end{proof}

\begin{rem} \label{rem:derive-rel}
As mentioned earlier in the paper, two of the
computer generated relations in Proposition
\ref{sl3-rest} are indicated by the calculation
above. For this we require the spectral 
sequence 
\[
E_2^{i,j} \ = \ \HHH^i(u(\fn/\fz), 
\HHH^j(u(\fz), k)) 
\Rightarrow \HHH^{i+j}(u(\fn),k)
\]
where $\fn$ is the 5-dimensional Lie algebra as
above and $\fz$ is the one-dimensional subalgebra
spanned by $u_5$. Note that $\fn/\fz \cong 
\fa \oplus \fb$ where $\fa$ (generated by the 
classes of $u_1, u_2, u_4$) is the nilpotent
radical of a Borel subalgebra of $\sl_3$, and 
$\fb$ (generated by $u_3$) is a one dimensional 
Lie algebra. The bottom row of the spectral 
sequence is generated by elements $\eta_\alpha, \eta_\beta, 
\eta_{2\alpha+\beta}, \eta_{\alpha+2\beta}, 
X_\alpha, X_\beta,X_{\alpha+\beta}$, generating 
$\HHH^*(u(\fa),k)$, and $\eta_\gamma, X_\gamma$, generating
$\HHH^*(u(\fb),k)$. The weights are as indicated 
by the subscripts.
The term $E_2^{1,0}$ is spanned by an element 
$\eta_{\beta+\gamma}$. Its image   
under the differential $d_2$ is $\eta_\beta \eta_\gamma$. 
What we know from the proof of Proposition~\ref{prop:notcm1} is 
that $\eta_\alpha \eta_\gamma X_\beta = 0$. 
This element can only be zero if it is in the 
image of $d_2$. Then by an examination of 
weights we see that (up to some nonzero 
scalar multiple) $d_2(\eta_{\beta+\gamma}\eta_{\alpha+2\beta}) = 
\eta_\beta \eta_\gamma \eta_{\alpha+2\beta} = 
\eta_\alpha \eta_\gamma X_\beta$. Note that this relation 
occurs in the ring $E_2^{*,0} \cong 
\HHH^*(u(\fa),k) \otimes \HHH^*(u(\fb),k)$.
Consequently, we must have that 
$\eta_\beta \eta_{\alpha+2\beta} = 
\eta_\alpha X_\beta$ in $\HHH^*(u(\fa),k)$, as asserted. 
The relation $\eta_\alpha \eta_{2\alpha+\beta} 
= \eta_\beta X_\alpha$ follows by 
symmetry (interchanging $u_1$ with $u_3$
and $u_4$ with $u_5$). 
\end{rem}

We shall see in Section \ref{sec:metabelian} that the 
example in Proposition \ref{prop:notcm1}
can be generalized, giving a metabelian Lie algebra
whose cohomology ring is not Cohen-Macaulay
for any prime $p$. 

%%%%%%%%%%%%%%%%%%%%  sec type B  %%%%%%%%%%%%%%%%%%%

\section{Some examples of type B}
In this section, we consider the nilpotent radical $\fn$
of the Borel subalgebra of a Lie algebra of type 
$B_2$ and some extensions thereof. We show that in 
characteristic ~5, the cohomology of $\fn$ has 
the form $\HHH^*(u(\fn),k) \cong
S(\fn^*)^{(1)} \otimes \HHH^*(\fn,k)$
as a module over $S(\fn^*)^{(1)}$, but not as a ring. 
Moreover, there is a one-dimensional extension of 
this Lie algebra whose cohomology is not Cohen-Macaulay. 
The basic idea of the construction applies other 
non-simply laced cases in other characteristics. 

Note that if the characteristic of $k$ is greater than
$6 = 2(h-1)$ then the isomorphism $\HHH^*(u(\fn),k) \cong
S(\fn^*)^{(1)} \otimes \HHH^*(\fn,k)$ is an 
isomorphism of rings by \cite[Theorem 3.1.1]{DNN}. 

The Lie algebra $\fn$ has a basis $u_1, u_2, u_3, u_4$
and the Lie bracket is given by $[u_1,u_2] = u_3$, 
$[u_1, u_3] = u_4$, $[u_2,u_3] = 0$ with $u_4$ being 
central.  There is an action of a two-dimensional
torus, $T$, relative to which the basis elements
have weights $-\alpha, -\beta, -\alpha-\beta, -2\alpha-\beta$
respectively ($\alpha$ being the short simple root).

We begin with a calculation of the ordinary 
Lie algebra cohomology of $\fn$. As in the last section, we 
let the subscripts of the elements denote their weights. 

\begin{lemma}\label{lem:ordcohob2}
The cohomology ring $\HHH^*(\fn,k)$ is generated 
by elements which we denote $\eta_\alpha,  \eta_\beta,
\eta_{\alpha+2\beta}, \eta_{3\alpha+\beta}, 
\eta_{4\alpha+2\beta}, \eta_{3\alpha+3\beta}$
in degrees 1,1,2,2,3,3.  All products 
of two of the given generators are zero except 
for the products 
\[
\eta_{\alpha} \eta_{3\alpha+3\beta} = 
-\eta_\beta\eta_{4\alpha+2\beta}
= \eta_{3\alpha+\beta} \eta_{\alpha+2\beta}
\] 
\end{lemma}

\begin{proof}
Let $\fz$ be the subalgebra spanned by $u_4$. 
Then we have a spectral sequence given as
\[
E_2^{i,j} = \HHH^i(\fn/\fz, 
\HHH^j(\fz,k)) \Rightarrow 
\HHH^{i+j}(\fn,k). 
\]
The bottom row is the cohomology of the nilpotent
radical of a Borel subalgebra of $\sl_3$, which 
is provided in Lemma \ref{sl3-ordin}. We adopt the 
notation for the elements as given in that lemma. 
The $E_2$ term of the spectral sequence is generated
as a ring by one additional element $\zeta = \zeta_{2\alpha+\beta} 
\in E_2^{0,1}$. Its image 
under the $d_2$ differential is the extension class
$\eta_{2\alpha+\beta}$ (having the same weight). 
The differential vanishes on every
product of $\zeta$ with an element on the bottom row
except $\eta_\beta \zeta$, and there $d_2(\eta_\alpha
\zeta) = \eta_\beta \eta_{2\alpha+\beta}
= -\eta_\alpha \eta_{\alpha+2\beta}$. 
The product structure is derived
from the product on the $E_2$ page as well as 
weight considerations. 
\end{proof}

Now we extend this to the restricted Lie algebra 
cohomology. Some of the relations in the proposition
given below were calculated using the basic 
algebra package in Magma. 

\begin{prop} \label{prop:b2rest}
Suppose that $k$ is a field of characteristic 5. 
Let $\fn$ be the nilpotent 
radical of the Borel subalgebra of a Lie
algebra of type $B_2$. The 
cohomology ring $\HHH^{*}(u(\fn),k) \cong S^{*}(\fn^*)^{(1)} 
\otimes \HHH^*(\fn, k)$ is a free 
$S^{*}(\fn^*)^{(1)}$-module with basis consisting 
of the images of a basis of the ordinary Lie algebra
cohomology of $\fn$ as in Lemma \ref{sl3-ordin}.
The ring $S^{*}(\fn^*)^{(1)}$ is a polynomial ring 
in variables $X_\alpha, X_\beta, X_{\alpha+\beta}$ 
and $X_{2\alpha+\beta}$ having weights
$5\alpha, 5\beta, 5(\alpha+\beta)$ and 
$5(2\alpha+\beta)$ under
the action of the torus. The multiplicative relations
among the generators of $1 \otimes \HHH^*(\fn,k) 
\subseteq \HHH^*(u(\fn),k)$ are exactly as given
in Lemma \ref{lem:ordcohob2}, except that 
$\eta_{3\alpha+\beta}^2 = 
\eta_{\alpha+2\beta}X_\alpha$. In particular, 
the isomorphism $\HHH^{*}(u(\fn),k) \cong S^{*}(\fn^*)^{(1)} 
\otimes \HHH^*(\fn, k)$ does not hold as rings. 
\end{prop}

\begin{proof}
The first statement follows from \cite[Theorem 3.1.1]{DNN}. A large
part of the remainder of the proof can be derived 
from the LHS spectral sequence and weight 
considerations. The unusual relation was verified by the computer. 
\end{proof}

Next we consider an extension of the algebra, whose
cohomology ring is not Cohen-Macaulay. As in the 
case of Remark \ref{rem:derive-rel}, we can derive
the unusual relation in the above proposition from
the calculations of the cohomology of the extension. 

Let $\fn$ be the restricted Lie algebra of dimension 
~5, with basis $u_1, \dots, u_5$ and Lie bracket 
given by $[u_1,u_2] = u_3$, 
$[u_1, u_3] = u_4$, $[u_1,u_4] = u_5$,
with $u_2, \dots, u_5$ forming a commutative subalgebra
that we denote $\fv$. We continue to assume that the 
characteristic of $k$ is ~5. The algebra $\fn$ has 
an action of a 2-dimensional torus, so that the elements
$u_1, \dots, u_5$ have weights $-\alpha$, $-\beta$, 
$-\alpha-\beta$, $-2\alpha-\beta$ and $-3\alpha-\beta$. 

\begin{prop}\label{prop:b2notcm}
The cohomology ring $\HHH^*(u(\fn),k)$ is not
Cohen-Macaulay. In particular, there is an element
in $\HHH^2(u(\fn),k)$ whose annihilator $\fP$ has
the property that 
$\HHH^*(u(\fn),k)/\fP$ has Krull dimension 
four. 
\end{prop}

\begin{proof}
We consider the spectral sequence: 
$E_2^{i,j} = \HHH^i(u(\fn/\fv), \HHH^j(u(\fv),k))\Rightarrow 
\HHH^{i+j}(u(\fn),k)$. As a module over $u(\fn/\fv)$
$M = \HHH^1(u(\fv),k)$ is uniserial of length ~4, 
generated by an element $\eta_{3\alpha + \beta}$.
Then $\HHH^2(u(\fv),k)$ is the exterior square of 
$M$ and it has a free $u(\fn/\fv)$-summand 
generated by $\eta_{3\alpha+\beta} \wedge 
u_{\alpha}\eta_{3\alpha+\beta}$
and having socle generated by $\eta_{\alpha+\beta}
\wedge \eta_\beta$. Hence there is
a class $\zeta$ in $E_2^{0,2}$ represented by a 
$u(\fn/\fv)$-homomorphism of $k$ onto the socle of this
summand. This class is annihilated on the $E_2$ page
of the spectral sequence by the class $X_\alpha$ 
in $E_2^{2,0}$. As in the case of Proposition 
\ref{prop:notcm1}, we can argue by weights that $\zeta$
is represented by a nonzero class $\hat{\zeta} \in 
\HHH^2(u(\fn),k)$ such that $X_\alpha \hat{\zeta} =0$,
and the annihilator of $\hat{w}$ is contained in a prime 
$\fP$ having the asserted properties. 
\end{proof}

\begin{rem}\label{rem:notorus}
We should note that the action of the torus is not 
required to show that the example is not Cohen-Macaulay.
If it were the case that one of the differentials
$d_2$ or $d_3$ failed to vanish on the class $\zeta$,
then we would have that $d_2(\zeta) = X_\alpha \mu$ for 
$\mu$ in $E_2^{0,1}$ or that $d_3(\zeta) = X_\alpha \mu$ 
for $\mu \in E_3^{1,0}$. In either case we would have
a class in $\HHH^1(u(\fn),k)$ that is annihilated
by $X_\alpha$. Similarly, there is no way to ungrade 
the spectral sequence to avoid having a large associate
prime in the cohomology ring. 
\end{rem}

\begin{rem}\label{rem:derive-rel2}
We remark, as in \ref{rem:derive-rel}, that the 
unusual relation $\eta_{3\alpha+\beta}^2 = 
\eta_{\alpha+2\beta} X_\alpha$
in Proposition \ref{prop:b2rest} can be derived
from the last Proposition. The proof is almost
exactly the same as in Remark~\ref{rem:derive-rel} and
we leave the details to the interested reader. 
\end{rem}

\begin{rem}\label{rem:b2other-primes}
The situation in Proposition \ref{prop:b2notcm} can
be extended to give examples for other primes. For 
suppose that $p > 5$, and let $\fn$ be the restricted
$p$-Lie algebra of dimension $n+1$ for some $n$ 
with $(p+3)/2 \leq n < p$, defined as follows.
A basis for $\fn$ consists of the elements $v, 
u_1, \dots, u_n$, where $u_1, \dots, u_n$ span a
commutative subalgebra, which we denote $\fv$. 
Then the product is given
by $[v,u_i] = u_{i+1}$ for $i = 1, \dots, n-1$, 
and $[v,u_n]= 0$. The $p$th-power operation is zero
on $\fn$. The algebra $\fn$ has an action of a 
two-dimensional torus such that the basis elements 
$v, u_1, \dots, u_n$ have weights $-\alpha, -\beta,
-\alpha-\beta, \dots, -(n-1)\alpha-\beta$, respectively. 
Then $M = \HHH^1(u(\fv), k)$ is an indecomposable 
uniserial module of dimension $n$ over the algebra 
$u(\fn/\fv)$. Because $n \geq (p+3)/2$, its 
exterior square $\Lambda^{2}(M)$ has a free summand.
Hence, considering the spectral sequence 
with $E_2$ term $E_2^{i,j} = \HHH^i(u(\fn/\fv),
\HHH^j(u(\fv),k)) \Rightarrow \HHH^{i+j}(u(\fn),k)$
and arguing exactly as in the proof of 
Proposition \ref{prop:b2notcm}, we get that 
$\HHH^*(u(\fn),k)$ has an associated prime $\fP$
such that  $\HHH^*(u(\fn),k)/\fP$ has Krull 
dimension at most $n$. 
\end{rem}

\section{An example of type $G_2$ in characteristic 7}
In this section we consider the nilpotent radical of a 
Borel subalgebra of the restricted Lie algebra of type $G_2$.
We obtain a similar result to that in Proposition 
\ref{sl3-rest} and Proposition \ref{prop:b2rest}.
Because the methods are also very similar to those in the
aforementioned propositions, we give only a sketch. 

\begin{prop}\label{prop:g2rest}
Suppose that $\fn$ is the nilpotent radical of a
Borel subalgebra of the restricted Lie algebra of type $G_2$.
Then as a module over the symmetric algebra $S^{*}(\fn^*)^{(1)}$,
we have that $\HHH^*(u(\fn),k) \cong S^{*}(\fn^*)^{(1)} 
\otimes \HHH^*(\fn,k)$. However, this is not an isomorphism
as rings. 
\end{prop}

\begin{proof} 
The characteristic of the field $k$ is larger than the 
Coxeter number and hence the first statement is a consequence of 
\cite[(3.5) Prop.]{FP}. Our task is to prove the second statement. 
Suppose that $\alpha$ and $\beta$ are the simple roots for 
the root system of type $G_2$. Assume that $\alpha$ is the 
short root. The other positive roots are $\alpha+\beta, 2\alpha+\beta,
3\alpha+\beta, 3\alpha+2\beta$.  We construct the central extension $\fg$
\[
\xymatrix{
0 \ar[r] & \fa \ar[r] & \fg \ar[r] & \fn \ar[r] & 0 
}
\]
where $\fa$ has dimension one, and $\fg$ has an action by 
the two dimensional torus so that an element of $\fa$ has 
weight $-4\alpha-\beta$. Thus $\fg$ has basis $u_\alpha,
u_\beta, u_{\alpha+\beta}, u_{2\alpha+\beta}, u_{3\alpha+\beta}, 
u_{4\alpha+\beta}, u_{3\alpha+2\beta}$ where the subscript on 
each element indicates the negative of its weight.

Let $\fv$ be the subalgebra generated by 
$u_\beta, u_{\alpha+\beta}, u_{2\alpha+\beta}, u_{3\alpha+\beta}, 
u_{4\alpha+\beta}, u_{3\alpha+2\beta}$, and let $\fz$ be the 
subalgebra generated by $u_{3\alpha+2\beta}$. The cohomology
of $\fv$ can be computed from the spectral sequence 
$E_2^{i,j} = \HHH^i(\fv/\fz, \HHH^j(\fz,k)) \Rightarrow 
\HHH^{i+j}(\fv,k)$.  Let $\eta_\gamma$ denote an element of 
weight $\gamma$ on this $E_2$ page. The differential must 
take the element $\eta_{3\alpha+2\beta} \in E_2^{0,1}$ 
to the extension class 
$d_2(\eta_{3\alpha+2\beta}) = \eta_\beta \wedge 
\eta_{3\alpha+\beta} - \eta_{\alpha+\beta} \wedge 
\eta_{2\alpha+\beta} \in E_2^{2,0}$. 
Note that this element is annihilated
by the action of $u_\alpha$, as is $\eta_{2\alpha+\beta}$.
The point of this calculation is that $\HHH^2(\fv,k)$ contains
a free module under the action of $u(\fg/\fv)$. That is the
element $\eta_{4\alpha+\beta} \wedge \eta_{3\alpha+\beta}$
generates a uniserial module of dimension $7$ over
$u(\fg/\fv)$, whose socle (the submodule annihilated by 
the action of $u_\alpha$) is spanned by $\eta_{\alpha+\beta}
\wedge \eta_{\beta}$, a class that survives to the $E_\infty$
page of the spectral sequence.  This implies that $\HHH^{*}(u({\mathfrak g}),k)$  is not Cohen-Macaulay. 
That is, we see from the spectral sequence $E_2^{i,j} = \HHH^i(u(\fg/\fv),
\HHH^j(u(\fv),k))\Rightarrow \HHH^{i+j}(u({\mathfrak g}),k)$, that the element $X_\alpha \in E_2^{0,2}$ in the 
symmetric algebra annihilates the element corresponding 
to $\eta_{\alpha+\beta} \wedge \eta_{\beta} \in E_2^{0,2}$.
This spectral sequence collapses at the $E_2$ page, and hence
the relation exists in $\HHH^*(u(\fg),k)$. 

The proposition is a consequence of Theorem \ref{thm:tensoralg}.
More specifically, by following the arguments in Proposition
\ref{sl3-rest} and using the weight information, we see that
in $\HHH^4(u(\fn),k)$ there must be a relation having 
roughly the form $(\eta_\alpha \wedge \eta_{3\alpha+\beta})^2
= (\eta_{\alpha+\beta} \wedge \eta_{\beta}) X_\alpha$. Note that
both are elements of weight $8\alpha+2\beta$.
\end{proof}

%%%%%%%%%%%%%%%%%%%  section 7  %%%%%%%%%%%%%%%%%%%%%%%%%%%%%%%%

\section{A metabelian example} \label{sec:metabelian}
In this section we present an example of a metabelian restricted
Lie algebra with the property that its cohomology ring is not 
Cohen-Macaulay. Such an example can be constructed for any value
of $p\geq 3$, except that the dimension of the example depends on 
the prime $p$.

Let $\fn$ be the nilpotent restricted Lie algebra with basis
consisting of the elements $u, v_i, w_i$ for $i = 1, \dots,
n$, and Lie bracket defined
by the rule 
\[
[u,v_i] = w_i, \quad [u,w_i] = 0 = [v_i, v_j] = [v_i, w_j] = 
[w_i, w_j]
\]
for all $i,j$ such that $1 \leq i, j, \leq n$. Note that $\fn$ 
is isomorphic to a subalgebra of $\sl_{n+2}$. That is, we can 
define a homomorphism $\varphi: \fn \to \sl_{n+2}$ as follows. 
Let $E_{i,j}$ be the matrix with $1 \in k$ in the $(i,j)$ 
position and 0 elsewhere. Then define $\varphi$ by $\varphi(u)
= E_{1,2}$, $\varphi(v_i) = E_{2,i+2}$ and $\varphi(w_i) = 
E_{1,i+2}$. The image of $\varphi$ has the property that the 
$p^{th}$-power of any element in the algebra is zero, since
$p \geq 3$.  Also, we note that the algebra has an action of 
the diagonal torus of $\sl_{n+2}$ of dimension $n+1$. 

Let $\fv$ be the subalgebra with basis consisting of all of the
elements $v_i, w_i$ for $i = 1, \dots, n$. This subalgebra is
commutative. We consider the 
spectral sequence $E_2^{r,s} = 
\HHH^r(u(\fn/\fv), \HHH^s(u(\fv), k) \Rightarrow
\HHH^{r+s}(u(\fn),k)$. The cohomology group $\HHH^1(u(\fv),k) =
\HHH^1(\fv,k)$ has dimension $2n$ and is spanned by 
elements $\gamma_i$ (of weight $\alpha_2 +\dots + \alpha_{i+1}$)
and $\eta_i$ (of weight $\alpha_1 + \dots + \alpha_{i+1}$,
for $i = 1, \dots, n$. The action of the element $u \in \fn/\fv$ on 
$\HHH^1(u(\fv),k)$ is given by $u \cdot \eta_i = \gamma_i$ and 
$u \cdot \gamma_i = 0$. Thus, $\HHH^1(u(\fv),k)$ is a direct sum 
of $n$ uniserial $u(\fn/\fv)$-modules of dimension 2. 

With this information, we can prove the following.

\begin{prop}\label{prop:metabelian}
Let $n = p-1$. Then the cohomology ring $\HHH^*(u(\fn),k)$ is
not Cohen-Macaulay.
\end{prop}

\begin{proof}
Let $M$ denote the uniserial $u(\fn/\fv)$-module of dimension 2. As noted
$E_2^{0,1} = \HHH^1(u(\fv),k)$ is a directs sum of $p-1$ copies of 
$M$. Hence, $E_2^{0,p-1}$ contains the $p-1$ exterior power of 
$\HHH^1(u(\fv),k)$ which includes the $p-1$ tensor power of $M$. 
The $p-1$ tensor power of $M$ has a projective $u(\fn/\fv)$-module, 
the uniserial module of dimension $p$ generated by 
$\eta_1 \wedge \dots \wedge \eta_{p-1}$, and having socle 
spanned by $\gamma_1 \wedge \dots \wedge \gamma_{p-1}$. Consequently,
$E_2^{0,p-1}$ has an element $y$ of weight 
$\alpha_2 + \dots + \alpha_{p}$.
This element must survive to the $E_{\infty}$ page of the spectral 
sequence, because there is no element of the same 
weight in $E_r^{p-1-r,r+2}$
for any value of $r>0$.  On the other hand, because 
$\gamma_1 \wedge \dots \wedge \gamma_{p-1}$ is in the socle of a 
projective $u(\fn/\fv)$-module, we must have that $X.y = 0$ 
for $X$ the generator of the symmetric algebra in $E_{2.0}$. 
Again, by a weight argument we see that the product $X.y$ cannot
ungrade to an element that is not zero. Consequently, this 
relation must exist in $\HHH^*(u(\fn).k)$. This implies that $X$
is not a regular element and that the depth of $\HHH^*(u(\fn),k)$
is less than the Krull dimension.  
\end{proof}

\section{Nilpotent Lie algebras of dimension $\leq 5$}

In this section we will use deGraaf's classification 
of indecomposable nilpotent Lie algebras over 
fields of characteristic not equal to 2. Again our interest is in 
whether there is an isomorphism 
\[
\HHH^*(u({\fn}),k)\cong 
S^{*}({\fn}^{*})^{(1)}\otimes 
\Lambda^{*}({\fn}^{*})
\cong S^{*}({\fn}^{*})^{(1)} \otimes 
\HHH^{*}({\fn},k)
\]
as rings or as modules over the symmetric algebra. If the algebra
is commutative, then the isomorphism holds as rings. On the other
hand if the $p$-power operation on the Lie algebra fails to vanish 
($x^{[p]} \neq 0$ for some $x$ in $\fn$), then the above isomorphism
can not hold. 

Given a finitel dimensional nilpotent Lie algebra, it is not always
possible to make it a restricted Lie algebra with trivial $p$-power
operation. For one thing, the adjoint action of any element on the 
algebra would have to be nilpotent of degree less than $p$.
We have listed below the indecomposable 
non-abelian nilpotent Lie algebras of 
dimension less than or equal to 5 along with the restrictions
on the prime $p$ that are necessary to impose a trivial $p$-power
operation. In this case, our standing assumption 
that ${\mathcal N}_{1}({\fn})\cong {\fn}$ holds.  The notation for
the Lie algebras is taken from DeGraaf's list \cite{deG}. 
\vskip .25cm 
\noindent
Dimension 3: \ $L_{3,2}$ ($p\geq 3$)
\vskip .5cm 
\noindent
Dimension 4: \ $L_{4,3}$ ($p \geq 5$)
\vskip .5cm 
The Lie algebras $L_{3,2}$ (reps. $L_{4,3}$) arise naturally 
as the unipotent radicals of the Borel subalgebras of 
simple Lie algebras of type $A_{2}$ (resp. $B_{2}$). 
We have $L_{3,2}=\langle x_{-\alpha}, x_{-\beta}, x_{-\alpha-\beta} 
\rangle$ and $L_{4,3}=\langle x_{-\alpha}, x_{-\beta}, 
x_{-\alpha-\beta}, x_{-2\alpha-\beta} \rangle$. The restricted cohomology rings
of these algebras are given in Propositions \ref{sl3-rest} and 
\ref{prop:b2rest}.
\vskip .5cm 
\noindent
Dimension 5:  $L_{5,4}$ ($p\geq 3$), $L_{5,5}$ ($p\geq 5$), 
$L_{5,6}$ ($p\geq 5$), $L_{5,7}$ $(p\geq 5$), 
$L_{5,8}$ ($p\geq 3$), $L_{5,9}$ ($p\geq 5$).
\vskip .5cm 
We now use deGraaf's description of the five dimensional 
nilpotent Lie algebra (cf. \cite[Section 4]{deG}) and 
describe natural gradings on these 
Lie algebras. The natural gradings are induced by toral 
actions given by outer automorphisms. When we compute
the cohomology of the algebras, differentials in the spectral sequences
respect the actions of these tori. 
\vskip .25cm 
\noindent
$L_{5,4}$: This nilpotent Lie algebra arises as a 
subalgebra of the nilpotent radical of a simple Lie 
algebra of type $A_{3}$. Let
$\alpha_{1},\alpha_{2},\alpha_{3}$ denote the 
simple roots. The $L_{5,4}$ consists of the span of 
the root vectors $\{x_{-\alpha_{1}},x_{-\alpha_{3}},
x_{-\alpha_{1}-\alpha_{2}},x_{-\alpha_{2}-\alpha_{3}},
x_{-\alpha_{1}-\alpha_{2}-\alpha_{3}}\}$. 
\vskip .25cm 
\noindent 
$L_{5,5}$: This Lie algebra has a double grading with 
basis $\langle x_{-\alpha},x_{-\beta},x_{-\alpha-\beta},
x_{-2\alpha-\beta},x_{-2\alpha} \rangle$. 
\vskip .25cm 
\noindent
$L_{5,6}$: Let $W(1)=\langle e_{i}:\ i\in {\mathbb Z}\}$ 
be the Witt algebra defined over ${\mathbb Z}$ with 
Lie bracket $[e_{i},e_{j}]=(i+j)e_{i+j}$. One can 
consider the subalgebra ${\mathfrak a}=\langle e_{i}:\ i< 0\}$ 
and factor this out by the ideal 
${\mathfrak z}=\{e_{i}:\ i\leq -6\}$. The Lie algebra 
$L_{5,6}$ is the Lie algebra ${\mathfrak a}/{\mathfrak z}$ 
tensored by $k$. This has a natural ${\mathbb Z}$ grading, 
thus an action of a one-dimensional torus 
on $L_{5,6}$. This Lie algebra can also be viewed as a non-graded 
central extension of the unipotent radical of type $B_{2}$. 
\vskip .25cm 
\noindent 
$L_{5,7}$: The Lie algebra $L_{5,7}$ is a graded central 
extension of $L_{4,3}$. The Lie algebra can be graded by a two-dimensional torus and has 
basis given by $\langle x_{-\alpha},x_{-\beta},x_{\alpha-\beta},
x_{-2\alpha-\beta},x_{-3\alpha-\beta} \rangle$. 
\vskip .25cm 
\noindent 
$L_{5,8}$: Let ${\mathfrak a}$ be the nilpotent radical for the Borel subalgebra of 
a Lie algebra of type $A_{3}$, and ${\mathfrak z}$ be the 
center of this Lie algebra. The Lie algebra 
$L_{5,8}$ can be realized as ${\mathfrak a}/{\mathfrak z}$ 
and has basis (with an action of a three dimensional torus) given by 
$\langle x_{-\alpha_{1}},x_{-\alpha_{2}},x_{-\alpha_{3}},
x_{-\alpha_{1}-\alpha_{2}},x_{-\alpha_{2}-\alpha_{3}} \rangle$. 
\vskip .25cm 
\noindent 
$L_{5,9}$: One can realize this Lie algebra as another 
graded central extension of $L_{4,3}$. This Lie 
algebra has a double grading with basis 
$\langle x_{-\alpha},x_{-\beta},x_{-\alpha-\beta},
x_{-2\alpha-\beta}, x_{-\alpha-2\beta} \rangle$. 
\vskip .5cm 
The ordinary Lie algebra cohomology can be computed recursively 
using central extensions and the LHS spectral sequence. For example, 
if ${\mathfrak a}$ is a nilpotent Lie algebra and 
${\mathfrak z}$ is a one-dimensional central 
subalgebra then the LHS spectral sequence: 
$$E_{2}^{i,j}=\operatorname{H}^{i}({\mathfrak a}/{\mathfrak z},k)
\otimes \operatorname{H}^{j}({\mathfrak z},k)
\Rightarrow \operatorname{H}^{i+j}({\mathfrak a},k)$$ 
will converge after the second page 
(i.e., $E_{3}\cong E_{\infty}$). We have 
$$\operatorname{H}^{2}({\mathfrak a},k)\cong 
\operatorname{H}^{2}({\mathfrak a}/{\mathfrak z},k)/\langle 
\text{Im }\delta_{2} \rangle \oplus \langle 
\text{Ker }\hat{\delta}_{2} \rangle.$$
where $\delta_{2}:E_{2}^{0,1}\rightarrow E_{2}^{2,0}$ 
and $\hat{\delta_{2}}:E_{2}^{1,1}\rightarrow E_{2}^{3,0}$. 
With appropriate choices of central subalgebras one 
can guarantee that the differentials respect the 
gradings above. This allows us to compute the 
differentials $\delta_{2}$ and $\hat{\delta}_{2}$ inductively. 
The weight spaces for the ordinary Lie algebra 
cohomology for these Lie algebras are multiplicity 
free (i.e., one-dimensional) and given in the 
following tables. 

To aid in the computation we use some facts 
about the ordinary Lie algebra 
cohomology. For example, that if $\fg$ has dimension
$d$, then the ordinary Lie algebra cohomology vanishes 
in degrees greater than $d$. In addition there is a 
Poincar\'e duality that is also respected by the action of 
the tori. So for example, if $d$ is the dimension of the algebra
$\fn$ and if the element in $\HHH^d(\fn,k)$ has weight $\gamma$, 
then the weights of the cohomology element in degrees $d-1$ will be 
$\gamma-\zeta_1, \gamma-\zeta_2, \dots $ where $\zeta_1,
\zeta_2, \dots$ are the weights of the cohomology elements in
degree 1.

\vskip.3in

\begin{table}[htbp]
\begin{tabular}{||r|r| |r|r||}

\hline

$L_{3,2}$ &  & $L_{4,3}$   &    \\
\hline
\hline
degree &  weights  & degree & weights    \\ 
\hline
0 & $0$                                                    &    0         &    $0$ \\
1 & $\alpha$, $\beta$                              &  1           & $\alpha$, $\beta$ \\
2  & $\alpha+2\beta$, $2\alpha+\beta$   &   2           & $\alpha+2\beta$,  $3\alpha+\beta$  \\
3   & $2\alpha+2\beta$                             &   3                 &   $3\alpha+3\beta$, $4\alpha+2\beta$ \\
     &                                                         &   4                  &  $4\alpha+3\beta$    \\
\hline
\end{tabular}
\end{table}

\begin{table}[htbp]
\begin{tabular}{||r|r||}

\hline

$L_{5,4}$ &      \\
\hline
\hline
degree &  weights    \\ 
\hline
0 & $0$                                                   \\
1 & $\alpha_{1}$, $\alpha_{3}$, $\alpha_{1}+\alpha_{2}$, $\alpha_{2}+\alpha_{3}$   \\
2  & $\alpha_{1}+\alpha_{3}$, $2\alpha_{1}+\alpha_{2}$, $\alpha_{1}+\alpha_{2}+\alpha_{3}$, \\
    & $\alpha_{2}+2\alpha_{3}$, $\alpha_{1}+2\alpha_{2}+\alpha_{3}$  \\
3   & $2\alpha_{1}+3\alpha_{2}+2\alpha_{3}$, $\alpha_{1}+2\alpha_{2}+3\alpha_{3}$, $2\alpha_{1}+2\alpha_{2}+2\alpha_{3}$,   \\
     & $3\alpha_{1}+2\alpha_{2}+\alpha_{3}$, $2\alpha_{1}+\alpha_{2}+2\alpha_{3}$ \\
4   &  $2\alpha_{1}+3\alpha_{2}+3\alpha_{3}$, $3\alpha_{1}+3\alpha_{2}+2\alpha_{3}$, $2\alpha_{1}+2\alpha_{2}+3\alpha_{3}$, \\ 
& $3\alpha_{1}+2\alpha_{2}+2\alpha_{3}$     \\
5   &  $3\alpha_{1}+3\alpha_{2}+3\alpha_{3}$     \\ 
\hline
\end{tabular}
\end{table}

\begin{table}[htbp]
\begin{tabular}{||r|r||}

\hline

$L_{5,8}$ &      \\
\hline
\hline
degree &  weights    \\ 
\hline
0 & $0$                                                   \\
1 & $\alpha_{1}$, $\alpha_{2}$, $\alpha_{3}$    \\
2  & $\alpha_{1}+2\alpha_{2}$, $2\alpha_{1}+\alpha_{2}$, $\alpha_{1}+\alpha_{3}$, \\
    &   $2\alpha_{2}+\alpha_{3}$, $\alpha_{2}+2\alpha_{3}$  \\
3   & $\alpha_{1}+\alpha_{2}+2\alpha_{3}$, $2\alpha_{2}+2\alpha_{3}$, $\alpha_{1}+3\alpha_{3}+\alpha_{3}$, \\
     & $2\alpha_{1}+\alpha_{2}+\alpha_{3}$, $2\alpha_{1}+2\alpha_{2}$  \\
4   & $\alpha_{1}+3\alpha_{2}+2\alpha_{3}$, $2\alpha_{1}+2\alpha_{2}+2\alpha_{3}$, $2\alpha_{1}+3\alpha_{2}+\alpha_{3}$                                                           \\
5   &  $2\alpha_{1}+3\alpha_{2}+2\alpha_{3}$ \\                                                        
\hline
\end{tabular}
\end{table}

\begin{table}[htbr]
\begin{tabular}{||r|r| |r|r||}

\hline

$L_{5,5}$ &  & $L_{5,6}$   &    \\
\hline
\hline
degree &  weights  & degree & weights    \\ 
\hline
0 & $0$                                                    &    0         &    $0$ \\
1 & $\alpha$, $\beta$, $2\alpha$                              &  1           & $1$, $2$ \\
2  & $2\alpha+\beta$, $\alpha+2\beta$, $3\alpha$   &   2           & $3$  \\
3   & $4\alpha+2\beta$, $5\alpha+\beta$, $3\alpha+3\beta$      &   3                 &   $12$  \\
4     &    $5\alpha+3\beta$, $6\alpha+2\beta$, $4\alpha+3\beta$                                                     &   4                  &  $13$, $14$    \\
5    & $6\alpha+3\beta$                       &   5       & $15$ \\
 \hline
\end{tabular}
\end{table}

\begin{table}[htbr]
\begin{tabular}{||r|r| |r|r||}
\hline
$L_{5,7}$ &  & $L_{5,9}$   &    \\
\hline
\hline
degree &  weights  & degree & weights    \\ 
\hline
0 & $0$                                                    &    0         &    $0$ \\
1 & $\alpha$, $\beta$                         &  1           & $\alpha$, $\beta$ \\
2  & $\alpha+2\beta$, $4\alpha+\beta$   &   2           & $3\alpha+\beta$, $\alpha+3\beta$  \\
3   & $6\alpha+2\beta$, $3\alpha+3\beta$      &   3                 &   $2\alpha+4\beta$, $4\alpha+2\beta$  \\
4     &    $6\alpha+4\beta$, $7\alpha+3\beta$            &   4                  &  $4\alpha+5\beta$, $5\alpha+4\beta$    \\
5    & $7\alpha+4\beta$                       &   5       & $5\alpha+5\beta$ \\
 \hline
\end{tabular}
\end{table}

\vfill\eject

With this information about the ordinary Lie algebra 
cohomology $\text{H}^{*}({\fn},k)$ we 
can deduce structural properties about the 
restricted Lie algebra cohomology. 

\begin{thm} Let ${\fn}$ be a five-dimensional 
nilpotent Lie algebra. Then 
\begin{itemize} 
\item[(a)] $\operatorname{H}^{*}(u({\mathfrak n}),k)
\cong S^{*}({\mathfrak n}^{*})^{(1)}\otimes 
\operatorname{H}^{*}({\mathfrak n},k)$ as 
$S^{*}({\mathfrak n}^{*})^{(1)}$ module for $L_{5,4}$ 
($p\geq 3$), $L_{5,5}$ ($p\geq 5$), $L_{5,6}$ 
($p\geq 7$), $L_{5,7}$ $(p\geq 7$), $L_{5,8}$ 
($p\geq 5$), and $L_{5,9}$ ($p\geq 5$). 
In these cases the cohomology ring 
$\HHH^{*}(u({\mathfrak n}),k)$ is Cohen-Macaulay. 
\item[(b)] In addition, the above isomorphism holds 
as rings for the algebras 
$L_{5,4}$ ($p\geq 5$), $L_{5,5}$ ($p\geq 7$), $L_{5,6}$ 
($p\geq 13$), $L_{5,7}$ $(p\geq 11$), $L_{5,8}$ 
($p\geq 5$), and $L_{5,9}$ ($p\geq 5$). 
\end{itemize}
\end{thm} 

\begin{proof} (a) There exist a one-dimensional central ideal ${\mathfrak z}$ such that ${\mathfrak n}/{\mathfrak z}$ 
 is isomorphic to (i) the nilpotent radical for a simple Lie algebra of Type $B_{2}$ for $L_{5,5}$, $L_{5,6}$, $L_{5,7}$ and $L_{5,9}$, (ii) an abelian Lie algebra 
for $L_{5,4}$ and (iii) the nilpotent radical of type $A_{2}\times A_{1}$ for $L_{5,8}$. 
Now assume that $p\geq 7$ in case (i), $p\geq 3$ in case (ii), and $p\geq 5$ in case (iii). Then 
$\operatorname{H}^{*}(u({\mathfrak n}/{\mathfrak z}),k)\cong 
S^{*}(({\mathfrak n}/{\mathfrak z})^{*})^{(1)}\otimes \operatorname{H}^{*}({\mathfrak n}/{\mathfrak z},k)$ as rings. Hence, by 
Theorem~\ref{thm:tensoralg}, $\operatorname{H}^{*}(u({\mathfrak n}),k)\cong 
S^{*}({\mathfrak n}^{*})^{(1)}\otimes 
\operatorname{H}^{\bullet}({\mathfrak n},k)$ as 
$S^{*}({\mathfrak n}^{*})^{(1)}$ module. 
 
The remaining cases when $L_{5,8}$ ($p=3$), $L_{5,5}$ ($p=5$) and $L_{5,9}$ ($p=5$) can be verified directly by showing that the 
spectral sequence (\ref{eq:specseq}) collapses. 

(b) We construct a subalgebra $B$ in 
$\text{H}^{*}(u({\mathfrak n}),k)$ which is 
isomorphic to $\text{H}^{*}({\mathfrak n},k)$. 
This is accomplished using the gradings to show that the 
following conditions cannot simultaneously hold. First,  
\[
\gamma_{1}+\gamma_{2}=\gamma_{3}+p\sigma
\]
where $\gamma_{j}$ is a weight of 
$\text{H}^{a_{j}}({\mathfrak n},k)$ for $j=1,2,3$ and $\sigma\neq 0$ 
is a weight of $S^{*}({\mathfrak n}^{*})^{(1)}$. Second, 
$$a_{1}+a_{2}=a_{3}+\text{deg}(\sigma)$$
were $\text{deg}(\sigma)$ is the cohomological degree of the element that $\sigma$ represents. 

In all the cases listed in the statement this was verified, and proves there do not exist 
two elements of weights $\gamma_1$ and $\gamma_2$ whose product is the product
of an element of the symmetric algebra with an element of weight
$\gamma_3$. Furthermore, this shows that a basis of weight vectors in 
$\text{H}^{*}({\mathfrak n},k)$ form
a subalgebra of $\text{H}^{*}(u({\mathfrak n}),k)$. 
\end{proof} 

\begin{thm} Let ${\mathfrak n}$ be a five-dimensional 
nilpotent Lie algebra. Then 
$\operatorname{H}^{*}(u({\mathfrak n}),k)$ 
is not Cohen-Macaulay in the cases that $\fn$ is 
$L_{5,7}$ for $p=5$ and $L_{5, 8}$ for $p=3$. In these cases, the depth of the 
cohomology ring is one less than the dimension. 
\end{thm} 

\begin{proof}
The results are proved in Propositions \ref{prop:notcm1} 
and \ref{prop:b2notcm}.
\end{proof}

We remark that the work in \cite{BC1, BC2} can be adapted for restricted Lie algebra cohomology to show that  
in the case when the cohomology ring has depth at least one less than the Krull dimension, the cohomology 
ring is Cohen-Macaulay if and only if it satisfies the functional equation in Theorem~\ref{thm:bc}. We can 
conclude that the cohomology rings for $L_{5,7}$ ($p=5$) and $L_{5, 8}$ ($p=3$) do not satisfy the 
functional equation.

\section{Type $A_{3}$, $p>h$} 

From the Examples \ref{sl3-rest}, \ref{prop:b2rest}, \ref{prop:g2rest},
one might get the impression that if ${\mathfrak n}$ is a 
unipotent radical of a Borel subalgebra of a simple Lie algebra 
and $h<p <2(h-1)$, then the restricted cohomology is not isomorphic
as rings to the tensor of the symmetric algebra and the ordinary Lie
algebra cohomology.  However, this is not true. In this section
we analyze the smallest example of this sort where the ring isomorphism
exists. In this case, the root system $\Phi$ is of type $A_{3}$ 
and ${\mathfrak n}$ is the six dimensional Lie algebra of 
strictly upper triangular $4\times 4$ matrices over 
a field of characteristic $5$. In the DeGraaf 
notation this is the Lie algebra $L_{6,19}(\epsilon)$. 

Since $p>h$, the ordinary Lie algebra cohomology is given by Kostant's 
theorem.  A proof in characteristic $p$ with $p>h$ 
is found in \cite[Theorem 4.1.1]{UGA}. As a module for $T$ we have 
\begin{equation} \label{eq:Kostant}
\operatorname{H}^{n}({\mathfrak n},k)\cong \bigoplus_{w\in W,\  l(w)=n} -w\cdot 0
\end{equation} 
As before we are using the convention that ${\mathfrak n}$ 
consists of negative root vectors. Let 
$\Delta=\{\alpha_{1},\alpha_{2},\alpha_{3}\}$ 
denote the simple root vectors. Then (\ref{eq:Kostant}) 
can be used to produce the following table which describes the weights in 
the cohomology groups $\operatorname{H}^{n}({\mathfrak n},k)$. 

\begin{table}[htbp]
\begin{tabular}{||r|r||}

\hline

$L_{6,19}(\epsilon)$ &      \\
\hline
\hline
degree &  weights    \\ 
\hline
0 & $0$                                                   \\
1 & $\alpha_{1}$, $\alpha_{2}$, $\alpha_{3}$    \\
2  & $\alpha_{1}+2\alpha_{2}$, $\alpha_{1}+\alpha_{3}$, $2\alpha_{1}+\alpha_{2}$, \\
    &   $\alpha_{2}+2\alpha_{3}$, $2\alpha_{2}+\alpha_{3}$  \\
3   & $2\alpha_{1}+2\alpha_{2}$, $\alpha_{1}+2\alpha_{2}+3\alpha_{3}$, $\alpha_{1}+3\alpha_{2}+\alpha_{3}$, \\
     & $2\alpha_{1}+\alpha_{2}+2\alpha_{3}$, $2\alpha_{2}+2\alpha_{3}$, $3\alpha_{1}+2\alpha_{2}+\alpha_{3}$  \\
4   & $2\alpha_{1}+2\alpha_{2}+3\alpha_{3}$, $2\alpha_{1}+4\alpha_{2}+2\alpha_{3}$, $\alpha_{1}+3\alpha_{2}+3\alpha_{3}$ \\
     & $3\alpha_{1}+3\alpha_{2}+\alpha_{3}$, $3\alpha_{1}+2\alpha_{2}+2\alpha_{3}$  \\
5   &  $2\alpha_{1}+4\alpha_{2}+3\alpha_{3}$, $3\alpha_{1}+3\alpha_{2}+3\alpha_{3}$, $3\alpha_{1}+4\alpha_{2}+2\alpha_{3}$ \\                                                        
6   &  $3\alpha_{1}+4\alpha_{2}+3\alpha_{3}$ \\                                                        
\hline
\end{tabular}
\end{table}

\begin{thm} Let ${\mathfrak n}$ be the unipotent radical 
corresponding to the simple group with root system $A_{3}$ 
(i.e., $4\times 4$ upper triangular matrices) with $p>h$. Then 
$$\operatorname{H}^{*}(u({\mathfrak n}),k)\cong 
S^{*}({\mathfrak n}^{*})^{(1)}\otimes 
\operatorname{H}^{*}({\mathfrak n},k)$$ 
as rings. 
\end{thm} 

\begin{proof} For $p>2(h-1)$ we can apply \cite[Theorem 3.1.1]{DNN}. 
This leaves us with the case when $p=5$. 
In order to invoke Theorem~\ref{th:splitting}, we need 
to analyze the equation 
\begin{equation} \label{eq:condition1}
-w_{1}\cdot 0-w_{2}\cdot 0=p\sigma-w_{3}\cdot 0
\end{equation}  
where $-w_{j}\cdot 0$ is a weight of 
$\text{H}^{a_{j}}({\mathfrak n},k)$ for $j=1,2,3$ and $\sigma$ 
is a weight of $S^{*}({\mathfrak n}^{*})^{(1)}$. 
Furthermore, by consideration
of cohomological degrees we must have that 
\begin{equation} \label{eq:condition2}
l(w_{1})+l(w_{2})-l(w_{3})=2\text{deg}(\sigma) 
\end{equation} 
where $\text{deg}(\sigma)$ is the cohomological degree 
corresponding to the element of weight $\sigma$. 

The first equation (\ref{eq:condition1}) does have solutions. For example, 
$$(2\alpha_{1}+4\alpha_{2}+3\alpha_{3})+(3\alpha_{1}+3\alpha_{2}+3\alpha_{3})
=(2\alpha_{2}+\alpha_{3})+5(\alpha_{1}+\alpha_{2}+\alpha_{3}).$$
Here $l(w_{1})=l(w_{2})=5$ and $l(w_{3})=2$. However, 
$\text{deg}(\sigma)\leq 6$. Thus, the second 
equation (\ref{eq:condition2}) cannot hold. A 
careful, case by case analysis rules out the possibility 
that both equations could simultaneously be satisfied. 
\end{proof} 

We end the paper with the following intriguing question. 

\begin{quest}
Suppose that $\fn$ is the nilpotent radical of a Borel 
subalgebra of a Lie algebra arising from a reductive algebraic group. 
In the case that the root system $\Phi$ is simply laced 
and that $p >h$, is there always a ring isomorphism 
$$\operatorname{H}^{*}(u({\mathfrak n}),k)\cong 
S^{*}({\mathfrak n}^{*})^{(1)}\otimes 
\operatorname{H}^{*}({\mathfrak n},k)?$$ 
\end{quest}

\vskip.3in

\end{document}